\documentclass[12pt,twoside]{amsart}
\usepackage{mathrsfs,amsthm,amscd,amsmath,amssymb}
\usepackage[colorlinks,linkcolor=green,citecolor=blue, pdfstartview=FitH]{hyperref}
\usepackage{amscd}
\usepackage{amsfonts}

\title[Some Koll\'ar-Enoki type injectivity and Nadel type vanishing theorems on compact \ka manifolds]
{Some Koll\'ar-Enoki type injectivity and Nadel type vanishing theorems on compact \ka manifolds}
\author{Chunle Huang}
\address{Chunle Huang, Institute of Mathematics, Hunan University, Changsha, 410082, China.}
\email{chunle@zju.edu.cn; 402961544@qq.com}

\newcommand{\Ker}[0]{\operatorname{Ker}}
\newcommand{\Image}[0]{\operatorname{Im}}

\newcommand{\ka}{K\"ahler }
\newcommand{\deldel}{\sqrt{-1}\partial \overline{\partial}}
\newcommand{\dbar}{\overline{\partial}}
\newcommand{\e}{\varepsilon}

\newcommand{\I}[1]{\mathcal{I}(#1)}

\newcommand{\lla}[0]{{\langle\!\hspace{0.02cm} \!\langle}}
\newcommand{\rra}[0]{{\rangle\!\hspace{0.02cm}\!\rangle}}
\newtheorem{thm}{Theorem}[section]
\newtheorem{lem}[thm]{Lemma}
\newtheorem{cor}[thm]{Corollary}

\theoremstyle{definition}
\newtheorem{defn}[thm]{Definition}

\newtheorem*{ack}{Acknowledgments}

\makeatletter
\let\uppercasenonmath\@gobble
\makeatother
\begin{document}
\bibliographystyle{amsalpha+}
\maketitle
\begin{abstract}
In this paper we will first show some Koll\'ar-Enoki type injectivity theorems on compact \ka manifolds, by using the Hodge theory, the Bochner-Kodaira-Nakano identity and the analytic method provided by O. Fujino and S. Matsumura in \cite{fujino-osaka,FM,matsumura4,matsumura7}. We have some straightforward corollaries. In particular, we will show that our main injectivity theorem implies several Nadel type vanishing theorems on smooth projective manifolds. Second, by applying the transcendental method, especially the Demailly-Peternell-Schneider equisingular approximation theorem and the H\"{o}rmander $L^2$ estimates, we will prove some Nakano-Demailly type and Nadel type vanishing theorems for holomorphic vector bundles on  compact \ka manifolds, twisted by pseudo-effective line bundles and multiplier ideal sheaves. As applications, we will show that our first main vanishing theorem generalizes the classical Nakano-Demailly vanishing theorem while the second one contains the famous Nadel vanishing theorem as a special case. 
\end{abstract}
\tableofcontents
\section{Introduction}\label{f-sec1}
The subject of cohomology vanishing theorems for holomorphic vector
bundles on complex manifolds occupies a role of central importance in
several complex variables and algebraic geometry (cf. \cite{EV86,esnault-viehweg,fujino-kodaira,fujino-vani-semi,fujino-amc,Kawamata}). Among various vanishing theorems 
the Kodaira vanishing theorem \cite{kodaira} 
is one of the most celebrated results 
in complex geometry and his original proof is based on his theory of 
harmonic integrals on compact K\"ahler manifolds.
The injectivity theorem as one of the most important generalizations of the Kodaira vanishing theorem plays an important role when we study fundamental problems in higher dimensional algebraic geometry (cf. \cite{fujino-pja,fujino-injectivity,fujino-zucker,matsumura1,matsumura2,matsumura6}). In particular, 
Koll\'ar obtained in \cite{kollar-higher1} his famous injectivity theorem, which is one of the most important generalizations of the Kodaira 
vanishing theorem for smooth complex projective varieties. 
After Koll\'ar's important work, 
Enoki recovered and 
generalized Koll\'ar's injectivity theorem in \cite{enoki} as an easy application of 
the theory of harmonic integrals on compact K\"ahler manifolds. 
Recently, O. Fujino and S. Matsumura in \cite{fujino-osaka,fujino-crelle,FM,matsumura3,matsumura4,matsumura5,matsumura7} have
obtained a series of important injectivity theorems on compact \ka manifolds formulated by singular hermitian metrics and multiplier ideal sheaves by using the transcendental method 
based on the theory of harmonic integrals on complete noncompact  K\"ahler manifolds. As is well known, the transcendental method often provides us some 
very powerful tools not only in complex geometry but also in algebraic geometry (cf. \cite{dhp,gongyo-matsumura,ohsawa-vanishing,ohsawa,ohsawa-takegoshi,paun,Siu98,Siu02}). Thus it is natural and of interest to 
study various vanishing theorems, injectivity theorems and other related topics 
by using the transcendental method. 
For a comprehensive and further description about this method, we recommend the reader to see the papers \cite{demailly,demailly-note,fujino-pja,fujino-funda,fujino-vanishing,fujino-injectivity,fujino-foundation,fujino-kollar-type} and also the references therein.

In this paper, we consider at first some Koll\'ar-Enoki type injectivity theorems on compact \ka manifolds by using the Hodge theory and the Bochner-Kodaira-Nakano identity on compact \ka manifolds and the analytic method provided by O. Fujino and S. Matsumura in \cite{fujino-osaka,FM,matsumura4,matsumura7}. Our first main result is the following Theorem \ref{chunle} 
which contains the famous Enoki injectivity theorem as a special case. 

\begin{thm}\label{chunle}
Let $L$ be a semi-positive holomorphic line bundle over a compact K\"ahler manifold $X$ with a smooth hermitian metric $h_L$ satisfying $\sqrt{-1}\Theta_{h_L} (L) \geq 0$. If $F$ $($resp.~$E$$)$ is a holomorphic line $($resp.~vector$)$ bundle over $X$ with a smooth hermitian 
metric $h_F$ $($resp.~ $h_E$$)$ such that 
\begin{enumerate}
\item $\sqrt{-1}\Theta_{h_F}(F) - a \sqrt{-1}\Theta_{h_L} (L) \geq 0$
\item $\sqrt{-1}\Theta_{h_E}(E)+(a-b)Id_E\otimes\sqrt{-1}\Theta_{h_L}(L)\geq_{Nak}0$\label{eq7} \,\,\text{in the sense of Nakano} \end{enumerate}
for some positive constants $a, b>0$, then for a nonzero section $s \in H^{0}(X, L)$ the multiplication map induced by $\otimes s$ 
\begin{equation*}
\times s:  
H^{q}(X, K_{X}\otimes E \otimes F ) 
\xrightarrow{} 
H^{q}(X, K_{X} \otimes E\otimes F \otimes L )
\end{equation*}
is injective for every $q\geq0$, where $K_X$ is the canonical line bundle of $X$. 
\end{thm}
Although the assumptions in Theorem \ref{chunle} may look a little bit artificial it is very useful and has some interesting applications. For instance, by applying Theorem \ref{chunle} we obtain the following Corollary \ref{en} and Corollary \ref{en_}. Corollary \ref{en} is just the original Enoki injectivity theorem. Corollary \ref{en_} generalizes the Enoki injectivity theorem to the case twisted by Nakano semi-positive vector bundles.

\begin{cor}[Enoki injectivity theorem cf. \cite{enoki,fujino-zucker,FM,matsumura4}]\label{en}
Let $L$ be a semi-positive line bundle over a compact K\"ahler manifold $X$. Then for a nonzero section $s \in H^{0}(X, L^l)$ 
the multiplication map induced by $\otimes s$
\begin{equation}
  \times s: H^q(X,K_X\otimes L^k)\rightarrow H^q(X,K_X\otimes L^{l+k})\nonumber
\end{equation}
is injective for any $k,l\geq1$ and $q\geq0$.
\end{cor}

\begin{cor}\label{en_}
Let $L$ $($resp.~ $E$$)$ be a semi-positive line bundle $($resp.~ a Nakano semi-positive vector bundle$)$ over a compact K\"ahler manifold $X$. Then for a nonzero section $s \in H^{0}(X, L^l)$ 
the multiplication map induced by $\otimes s$
\begin{equation}
  \times s: H^q(X,K_X\otimes E\otimes L^k)\rightarrow H^q(X,K_X\otimes E \otimes L^{l+k})\nonumber
\end{equation}
is injective for any $k,l\geq1$ and $q\geq0$.
\end{cor}

Motivated by the profound work obtained    
by O. Fujino and S. Matsumura in a series of papers (cf. \cite{fujino-osaka,FM,matsumura4,matsumura7}) we can generalize Theorem \ref{chunle} to the case formulated by singular hermitian metrics and multiplier ideal sheaves as follows.

\begin{thm}\label{f-thm5.1}
Let $L$ be a semi-positive holomorphic line bundle over a compact K\"ahler manifold $X$ with a smooth hermitian metric $h_L$ satisfying $\sqrt{-1}\Theta_{h_L} (L) \geq 0$.
If $F$ $($resp.~$E$$)$ is a holomorphic line $($resp.~vector$)$ bundle over $X$ with a singular hermitian 
metric $h$ $($resp.~a smooth hermitian metric $h_E$$)$ such that 
\begin{enumerate}
\item $\sqrt{-1}\Theta_{h}(F) - a \sqrt{-1}\Theta_{h_L} (L) \geq 0$ \,\,\text{in the sense of currents} 
\item $\sqrt{-1}\Theta_{h_E}(E)+(a-b)Id_E\otimes\sqrt{-1}\Theta_{h_L}(L)\geq_{Nak}0$\label{eq7} \,\,\text{in the sense of Nakano} \end{enumerate}
for some positive constants $a, b>0$, then for a nonzero section $s \in H^{0}(X, L)$ the multiplication map induced by $\otimes s$ 
\begin{equation*}
\times s:  
H^{q}(X, K_{X}\otimes E \otimes F \otimes \mathcal I(h)) 
\xrightarrow{} 
H^{q}(X, K_{X} \otimes E\otimes F \otimes \mathcal I(h) \otimes L )
\end{equation*}
is injective for every $q\geq0$, where 
$\mathcal I(h)$ is the multiplier ideal sheaf of $h$. 
\end{thm}
Here we remark that Theorem \ref{f-thm5.1} has many straightforward applications. For instance, by applying Theorem \ref{f-thm5.1} we have the following Corollary \ref{f-m} and  Corollary \ref{f-m2}. Corollary \ref{f-m} is the main injectivity theorem in \cite{FM} and Corollary \ref{f-m2} is the Theorem 6.6 in \cite{FM}.

\begin{cor}[Theorem A in \cite{FM}]\label{f-m}
Let $L$ be a semi-positive holomorphic line bundle over a compact K\"ahler manifold $X$ with a smooth hermitian metric $h_L$ satisfying $\sqrt{-1}\Theta_{h_L} (L) \geq 0$.
If $F$ is a holomorphic line bundle over $X$ with a singular hermitian 
metric $h$ such that $\sqrt{-1}\Theta_{h}(F) - a \sqrt{-1}\Theta_{h_L} (L) \geq 0$ in the sense of currents for some positive constants $a>0$, then for a nonzero section $s \in H^{0}(X, L)$ the multiplication map induced by $\otimes s$ 
\begin{equation*}
\times s: 
H^{q}(X, K_{X}\otimes F \otimes \mathcal I(h)) 
\xrightarrow{}
H^{q}(X, K_{X}\otimes F \otimes \mathcal I(h) \otimes L )
\end{equation*}
is injective for every $q\geq0$.  
\end{cor}

\begin{cor}[Theorem 6.6 in \cite{FM}]\label{f-m2}
Let $L$ be a semi-positive holomorphic line bundle over a compact K\"ahler manifold $(X,\omega)$ equipped with a smooth hermitian metric $h_L$ satisfying $\sqrt{-1}\Theta_{h_L} (L) \geq 0$ and $E$ a Nakano semi-positive vector bundle over $X$.
If $F$ is a holomorphic line bundle over $X$ with a singular hermitian 
metric $h$ such that $\sqrt{-1}\Theta_{h}(F) - a \sqrt{-1}\Theta_{h_L} (L) \geq 0$ for some positive constants $a>0$, then for a nonzero section $s \in H^{0}(X, L)$ the multiplication map induced by $\otimes s$ 
\begin{equation*}
\times s:
H^{q}(X, K_{X}\otimes E \otimes F \otimes \mathcal I(h)) 
\xrightarrow{} 
H^{q}(X, K_{X} \otimes E\otimes F \otimes \mathcal I(h) \otimes L )
\end{equation*}
is injective for every $q\geq0$.  
\end{cor}


Moreover, by applying Theorem \ref{f-thm5.1} we can also prove some vanishing theorems of Nadel type on smooth projective manifolds as follows.\begin{cor}\label{f-thm1.3}
Let $X$ be a smooth projective manifold with 
a K\"ahler form $\omega$ and $E$ a Nakano semi-positive vector bundle on $X$.
Let $F$ be a holomorphic line bundle on $X$ with 
a singular hermitian metric $h$ such that 
$\sqrt{-1}\Theta_{h}(F)\geq \varepsilon \omega$ in the sense of currents for some $\varepsilon >0$. Then for every $q>0$ we have 
$$H^q(X, K_X\otimes E\otimes F\otimes \mathcal I(h))=0.$$ 
\end{cor}
In particular we have 
\begin{cor}[{Nadel vanishing theorem due to 
Demailly:~\cite[Theorem 4.5]{demailly-numerical}}]\label{f-thm1.4}
Let $X$ be a smooth projective manifold with 
a K\"ahler form $\omega$ and $F$ be a holomorphic line bundle on $X$ with a singular hermitian metric $h$ such that 
$\sqrt{-1}\Theta_{h}(F)\geq \varepsilon \omega$ in the sense of currents for some $\varepsilon >0$. Then for every $q>0$ we have 
$$
H^q(X, K_X\otimes F\otimes \mathcal I(h))=0.
$$ 
\end{cor}

Although Corollary \ref{f-thm1.3} and Corollary \ref{f-thm1.4} can be derived by applying our main injectivity theorem (Theorem \ref{f-thm5.1}) we would like to give their other proof in the final section by making direct use of the transcendental method, especially the Demailly-Peternell-Schneider equisingular approximation theorem and the H\"{o}rmander $L^2$ estimates because we believe that it is natural and of much interest. In fact, by the transcendental method we can prove much more vanishing theorems on compact \ka manifolds not only on smooth projective manifolds. For instance we can show  

\begin{thm} \label{ND}
Let $(X,\omega)$ be a compact \ka manifold of dimension $n$,
$m$ a positive integer, $E$ a holomorphic vector bundle on $X$ of rank $r$ and $F$ a pseudo-effective line bundle on $X$ equipped with a singular hermitian metric $h$ with semi-positive curvature current.
If $E$ is Demailly $m$-positive then for any $q\geq 1$ with $m\geq\min\{n-q+1,r\}$ we have
\begin{equation}
H^q(X,K_X\otimes E\otimes F\otimes \mathcal{I}(h))=0 \nonumber
\end{equation}
where $K_X$ is the canonical bundle of $X$ and $\mathcal{I}(h)$
is the multiplier ideal sheaf of $h$.
\end{thm}
For the definition of the Demailly $m$-positivity see the following Section \ref{pre.}. Theorem \ref{ND} is the first main vanishing theorem in this paper, which generalizes the well-known Nakano-Demailly vanishing theorem to the case formulated by pseudo-effective line bundles and multiplier ideal sheaves. 
By Theorem \ref{ND} we obtain
the following Corollary \ref{.Griffiths} and Corollary \ref{.Nakano}.
Corollary \ref{.Griffiths} (resp. Corollary \ref{.Nakano}) generalizes the original Griffiths (resp. Nakano) vanishing theorem cf. \cite{Griffiths} (resp. cf. \cite{demailly-note,Nakano}).
\begin{cor} \label{.Griffiths}
Let $E$ be a holomorphic vector bundle 
over an $n$-dimensional compact \ka manifold $X$ and $F$ a pseudo-effective line bundle on $X$ equipped with a singular Hermitian metric $h$ with semi-positive curvature current.
If $E$ is Griffiths positive then 
\begin{equation}
H^n(X,K_X\otimes E\otimes F\otimes \mathcal{I}(h))=0.\nonumber
\end{equation}
\end{cor}
\begin{cor}  \label{.Nakano}
Let $E$ be a holomorphic vector bundle over a compact \ka manifold $X$ and $F$ a pseudo-effective line bundle on $X$ equipped with a singular Hermitian metric $h$ with semi-positive curvature current.
If $E$ is Nakano positive then for any $q\geq 1$ we have
\begin{equation}
H^q(X,K_X\otimes E\otimes F\otimes \mathcal{I}(h))=0.\nonumber
\end{equation}
\end{cor}
Here we present some other applications of Theorem \ref{ND}. 
Let $D=\sum\alpha_jD_j\geq 0$ be an effective $\mathbb{Q}$-divisor and define the multiplier ideal sheaf $\mathcal{I}(D)$ of $D$ to be equal to $\mathcal{I}(\varphi)$ where $\varphi=\sum\alpha_j\log|g_j|$ is the corresponding psh function
defined by generators $g_j$ of $\mathcal{O}(D_j)$.
Further, if we suppose that $D$ is a divisor with normal crossings, 
then we have 
$$\mathcal{I}(D)=\mathcal{O}(-[D])$$ where $[D]=\sum[\alpha_j]D_j$ is the integer part of $D$ (cf. \cite{demailly,dem}). As a simple consequence of Theorem \ref{ND} we obtain 

\begin{thm}\label{hehe}
Let $E$ be a holomorphic vector bundle of rank $r$ over an $n$-dimensional compact \ka manifold $X$, $m$ a positive integer.
Assume that $D=\sum^t_{i=1}a_iD_i$ is an effective $\mathbb{Q}$-divisor in X with normal crossings and denote $D'=\sum^t_{i=1}(a_i-[a_i])D_i$.
If $E$ is Demailly $m$-positive then for any $q\geq 1$ with $m\geq\min\{n-q+1,r\}$ we have
\begin{equation}
H^q(X,K_X\otimes E \otimes \mathcal{O}(D'))=0.\nonumber
\end{equation}
\end{thm}

In particular, we obtain  
\begin{cor} [cf. \cite{log}]\label{hdhd}
Let $E$ be a holomorphic vector bundle of rank $r$ over an $n$-dimensional compact \ka manifold $X$ and $D=\sum^t_{i=1}a_iD_i$ an effective normal crossing $\mathbb{Q}$-divisor $D$ in $X$ with $0\leq a_i<1$. If $E$ is Demailly $m$-positive then for any $q\geq 1$ with $m\geq\min\{n-q+1,r\}$ we have
\begin{equation}
H^q(X,K_X\otimes E\otimes D)=0.\nonumber
\end{equation}
\end{cor}

Our second main vanishing theorem is the following Theorem \ref{thm2}, which contains the famous Nadel vanishing theorem as a special case. 
By Theorem \ref{thm2} we obtain the following Corollary \ref{thm3} and Corollary \ref{Nadel}.

\begin{thm} \label{thm2}
Let $(X,\omega)$ be a compact \ka manifold of dimension $n$ and $L$ be a holomorphic line bundle on $X$ with  a singular hermitian metric $h$ such that $\sqrt{-1}\Theta_{h}(L)\geq\delta\omega$
for some constant $\delta>0$. If $(E,h_E)$ is an hermitian holomorphic vector bundle on $X$ of rank $r$ such that
\begin{equation}\nonumber
\sqrt{-1}\Theta_{h_E}(E)+\tau\omega\otimes Id_E\geq_{m}0
\end{equation}
for some constant $\tau<\delta$, then for any $q\geq 1$ with $m\geq\min\{n-q+1,r\}$ we have
\begin{equation}
  H^{q}(X,K_X \otimes E\otimes L\otimes\mathcal{I}(h))=0. \nonumber
\end{equation}
\end{thm}

\begin{cor} \label{thm3}
Let $(X,\omega)$ be a compact \ka manifold and $L$ a holomorphic line bundle on $X$ with a singular hermitian metric $h_L$ such that $i\Theta_{h}(L)\geq\delta\omega$
for some constant $\delta>0$. If $(E,h_E)$ is an hermitian holomorphic vector bundle on $X$ of rank $r$ such that
\begin{equation}\nonumber
\sqrt{-1}\Theta_{h_E}(E)+\tau Id_E\otimes\omega\geq_{Nak}0
\end{equation}
for some constant $\tau<\delta$, then for any $q\geq 1$ we have
\begin{equation}
  H^{q}(X,K_X \otimes E\otimes L\otimes\mathcal{I}(h))=0. \nonumber
\end{equation}
\end{cor}

\begin{cor} [Nadel vanishing theorem \cite{demailly,Nadel}]\label{Nadel}
Let $(X,\omega)$ be a compact K\"{a}hler manifold and $L$ be 
a holomorphic line bundle on $X$ with a singular hermitian
metric $h$ such that $i\Theta_{h}(L)\geq\delta\omega$
for some constant $\delta>0$. Then for any $q\geq1$ we have
\begin{equation}
  H^{q}(X,K_X \otimes L\otimes\mathcal{I}(h))=0.\nonumber
\end{equation}
\end{cor}
It is obvious that Corollary \ref{f-thm1.3} (resp. Corollary \ref{f-thm1.4})  follows from Corollary \ref{thm3} (resp. Corollary \ref{Nadel}). This means that we can give a direct proof of Corollary \ref{f-thm1.3} and Corollary \ref{f-thm1.4} by using the transcendental method rather than 
our main injectivity theorem (Theorem \ref{f-thm5.1}).

This paper is organizied as follows. 
In Section \ref{pre.}, we recall some basic definitions 
and collect several preliminary lemmas. 
Section \ref{poi.} is devoted to the 
proof of the main Koll\'ar-Enoki type injectivity theorems on compact \ka manifolds. We will give the proof of Theorem \ref{chunle} at first and then 
generalize Theorem \ref{chunle} to Theorem \ref{f-thm5.1} by applying the deep method provided by O. Fujino and S. Matsumura in \cite{fujino-osaka,FM,matsumura4,matsumura7}. 
We prove all the vanishing theorems in Section \ref{pov.}, in which a detailed proof of Theorem \ref{ND} and a short proof of Theorem \ref{thm2} will be given by the transcendental method based on the theory of harmonic integrals on complete noncompact  K\"ahler manifolds. 

\begin{ack}
The author would like to thank the referee
for carefully reading the paper and for valuable suggestions.
\end{ack}

\section{Preliminaries}\label{pre.}
In this section, we collect some basic definitions and results from complex analytic and differential geometry. For details, see, for example, \cite{demailly,demailly-note}.\\

\noindent\textbf{2.1. Positivity of vector bundles.} 
Let $E$ be a holomorphic vector bundle of rank $r$ over a complex
manifold $X$ and $h_E$ be a smooth hermitian metric on $E$. We know there exists a unique connection $\nabla$, called the Chern connection of
$(E,h_E)$, which is compatible with the metric $h_E$ and complex
structure on $E$. Let $\{z^i\}_{i=1}^n$ be the local holomorphic
coordinates on $X$ and $\{e_{\alpha}\}_{\alpha=1}^r$ be the local
holomorphic frames of $E$. Locally, the curvature tensor of $(E,h_E)$
takes the form
\begin{equation}
  \sqrt{-1} \Theta_{h_E}(E)=\sqrt{-1} R_{i\overline{j}\alpha}^\gamma dz^i\wedge d\overline{z}^j \otimes e^{\alpha}\otimes e_{\gamma}\nonumber
\end{equation}
where $R_{i\overline{j}\alpha}^\gamma=h^{\gamma\overline{\beta}}R_{i\overline{j}\alpha\overline{\beta}}$ and
$R_{i\overline{j}\alpha\overline{\beta}}=-\frac{\partial^2 h_{\alpha\overline{\beta}}}{\partial z^i\partial \overline{z}^j}
  +h^{\gamma\overline{\delta}}\frac{\partial h_{\alpha\overline{\delta}}}{\partial z^i}\frac{\partial h_{\gamma\overline{\beta}}}{\partial \overline{z}^j}.$
Here and henceforth we adopt the Einstein convention for summation.

\begin{defn} [\cite{demailly-note,SS}]
An hermitian vector bundle $(E,h_E)$ is said to be Griffiths-positive,
if for any nonzero vectors $u=u^i \frac{\partial}{\partial z^i}$ and $v=v^\alpha e_\alpha$,
$\sum_{i,j,\alpha,\beta}R_{i\overline{j}\alpha\overline{\beta}}u^i\overline{u}^j v^\alpha \overline{v}^\beta>0.$
$(E,h_E)$ is said to be Nakano-positive, if for any nonzero vector
$u=u^{i\alpha}\frac{\partial}{\partial z^i}\otimes e_\alpha$,
$\sum_{i,j,\alpha,\beta}R_{i\overline{j}\alpha\overline{\beta}}u^{i\alpha}\overline{u}^{j\beta}>0.$
\end{defn}

The notions of semi-positivity, negativity and
semi-negativity can be defined similarly. It is
clear that the Nakano positivity implies the Griffiths positivity and that
both concepts coincide if $r = 1$ (in the case of a line bundle, $E$
is merely said to be positive).
In \cite{demailly-dbar} Demailly introduced the notion of $m$-positivity
for any integer $1\leq m \leq r$ for a vector bundle $E$ of rank $r$
which interpolates between the Griffiths positivity and the Nakano positivity. We cited it as the Demailly $m$-positivity in this paper.
\begin{defn} [{\cite{demailly-dbar}}] \label{def.44}
Let $(E,h_E)$ be a hermitian vector bundle over a complex manifold $X$. A tensor $u\in TX\otimes E$ is called of rank $m$
if $m$ is the smallest non-negative integer
such that $u$ can be written as
$u=\sum^{m}_{j=1}\xi^{j}\otimes \upsilon^{j},\quad \xi^{j}\in TX,~~\upsilon^{j}\in E.$
$E$ is said to be Demailly $m$-positive (resp. $m$-semi-positve),
denoted by $E>_m 0$ (resp. $E\geq_m 0$),
if the hermitian form $\sqrt{-1}\Theta_{h_E}(E)(u,u)>0$ (resp. $\geq0$)
for any nonzero $u\in TX\otimes E$
of rank $\leq m$.
\end{defn}

It is obvious that the Demailly $1$-positivity is just the
Griffths-positivity and the Demailly $m$-positivity for
$m\geq \min\{r,n\}$ is exactly the Nakano-positivity (cf. \cite{demailly-note}, page 339). Here we abuse a little bit the notation and denote also
$\sqrt{-1}\Theta_{h_E}(E)$ to be the hermitian form associated
to the Chern curvature.\\

\noindent\textbf{2.2. Singular hermitian metrics and multiplier ideal sheaves.} 
Next let us recall the definition of singular hermitian metrics and its multiplier ideal sheaves. For the details, we recommend the reader to see \cite{demailly}. Let $F$ be a holomorphic line bundle
on a complex manifold $X$.
\begin{defn}\label{f-def2.1}
A singular hermitian metric
on $F$ is a metric $h_F$ which is given in
every trivialization $\theta: F|_{\Omega}\simeq \Omega\times
\mathbb C$ by $| \xi |_{h_F} =|\theta(\xi)|e^{-\varphi} \text{ on } \Omega,$
where $\xi$ is a section of $F$ on $\Omega$ and
$\varphi \in L^1_{\mathrm{loc}}(\Omega)$
is an arbitrary function.
Here $L^1_{\mathrm{loc}}(\Omega)$ is the space
of locally integrable functions on $\Omega$.
We usually call $\varphi$
the weight function of the
metric with respect to the trivialization $\theta$.
The curvature current of a singular hermitian metric $h_F$ is
defined by $\sqrt{-1}\Theta_{h_{F}}(F):=2\sqrt{-1}\partial\overline{\partial} \varphi,$
where $\varphi$ is a weight function and
$\partial\overline{\partial} \varphi$ is taken in the sense of distributions.
It is easy to see that the right hand side does not depend on
the choice of trivializations (cf. \cite{demailly}).
\end{defn}

\begin{defn}
A holomorphic line bundle $F$ is said to be pseudo-effective if $F$ admits a singular hermitian metric $h_F$ with semi-positive curvature current.
\end{defn}

The notion of multiplier ideal sheaves introduced by Nadel in \cite{Nadel}
is very important in the recent developments of complex geometry and algebraic geometry.

\begin{defn}\label{f-def2.2}
A quasi-plurisubharmonic function by definition is a function $\varphi$ which is locally equal to the sum of a plurisubharmonic function and of a smooth function. If $\varphi$ is a quasi-plurisubharmonic function on a complex manifold $X$, then the multiplier ideal sheaf $\mathcal J(\varphi)\subset \mathcal O_X$ is defined by
\begin{equation*}
\Gamma (U, \mathcal J(\varphi))
:=\{f\in \mathcal O_X(U)\, |\, |f|^2e^{-2\varphi}\in
L^1_{\mathrm{loc}}(U) \}
\end{equation*}
for every open set $U\subset X$.
Then it is known that
$\mathcal J(\varphi)$ is a coherent ideal sheaf of $\mathcal O_X$
(see \cite[(5.7) Lemma]{demailly} for example).
\end{defn}

\begin{defn} \label{f-def2.3}
Let $F$ be a holomorphic line bundle over a complex manifold $X$ and
let $h_F$ be a singular hermitian metric on $F$.
We assume $\sqrt{-1}\Theta_{h_F}(F)\geq \gamma$
for some smooth $(1, 1)$-form $\gamma$ on $X$.
We fix a smooth hermitian metric $h_{\infty}$ on $F$.
Then we can write $h_{F}=h_{\infty} e^{-2\psi}$ for some
$\psi \in L^1_{\mathrm{loc}}(X)$ and $\psi$ coincides with a quasi-plurisubharmonic function $\varphi$ on $X$ almost everywhere.
In this situation, we put $\mathcal J(h_{F}):=\mathcal J(\varphi)$.
We note that $\mathcal J(h_{F})$ is independent of
$h_{\infty}$ and is thus well-defined.
\end{defn}

\noindent\textbf{2.3. Equisingular approximations.}    
The following Lemma \ref{equi.} is the well-known Demailly-Peternell-Schneider equisingular approximation theorem, which is frequently used in this paper. For details, see \cite[Theorem 2.3]{dps} and \cite[Theorem 2.3]{matsumura4}.
\begin{lem}\label{equi.}
Let $F$ be a holomorphic line bundle on a compact \ka manifold $(X,\omega)$ with a singular hermitian metric $h$ with semi-positive curvature current. Then exists a countable family $\{h_{\e} \}_{1\gg \e>0}$ of singular hermitian metrics on $F$ with the following properties: 
\begin{itemize}
\item[(a)]$h_{\e}$ is smooth on $Y_{\e}:=X \setminus Z_{\e}$, 
where $Z_{\e}$ is a proper closed subvariety on $X$. 
\item[(b)]$h_{\e'} \leq h_{\e''} \leq h$ holds on $X$ 
when $\e' > \e'' > 0$.
\item[(c)]$\mathcal I({h})= \mathcal I({h_{\e}})$ on $X$.
\item[(d)]$\sqrt{-1} \Theta_{h_{\e}}(F) 
\geq a\sqrt{-1}\Theta_{h_{F}}(F) -\e \omega$ on $X$. 
\end{itemize}
\end{lem}

\noindent\textbf{2.4. $L^2$ spaces and $L^2$ estimates.} 
Let $X$ be a complex manifold with a positive $(1,1)$-form $\omega$
and $E$ be a holomorphic vector bundle over $X$ with a smooth metric $h$. For $E$-valued $(p,q)$-forms $u$ and $v$,
the point-wise inner product
$\langle u, v\rangle _{h, \omega}$
can be defined, and
the global inner product
$\lla u, v \rra  _{h, \omega}$
can also be defined by
\begin{equation*}
\lla u, v \rra  _{h, \omega}:=
\int_{X}
\langle u, v\rangle _{h, \omega}\, dV_{\omega}
\end{equation*}
where $dV_{\omega}:= \omega^{n}/n!$ and $n$ is the dimension of $X$. Recall that
the Chern connection $D_{h}$ on $E$ determined by the holomorphic structure and
the hermitian metric $h$ can be written as
$D_{h} = D'_{h} + \bar\partial$ with the $(1,0)$-connection $D'_{h}$ and the
$(0,1)$-connection $\bar\partial$ (the $\bar\partial$-operator).
The connections $D'_{h}$ and $\bar\partial$
can be regarded as a densely defined closed operator on
the $L^{2}$-space $L_{(2)}^{p, q}(X, E)_{h, \omega}$ defined by
\begin{equation*}
L_{(2)}^{p, q}(X, E)_{h, \omega}:=
\{u \mid u \text{ is an }E\text{-valued }(p, q)\text{-form such that }
\|u \|_{h, \omega}< \infty \}.
\end{equation*}
The formal adjoints $D'^{*}_{h}$ and $ \bar\partial^{*}_{h}$ agree with
the Hilbert space adjoints in the sense of Von Neumann
if $\omega$ is a complete metric on $X$.
For the $L^2$-space $L^{p,q}_{(2)}(X, E)_{h, \omega}$
of $E$-valued $(p,q)$-forms on $X$ with respect to the inner product $\|\bullet \|_{h, \omega}$,
we define the $L^2$ cohomology $H^{p,q}_{(2)}(X, E)_{h, \omega}$ by
$$
H^{p,q}_{(2)}(X, E)_{h, \omega}:=
\frac{{\rm{Ker}}\, \bar\partial \cap L^{p,q}_{(2)}(X, F)_{h, \omega}}{{\rm{Im}\, \bar\partial}\cap L^{p,q}_{(2)}(X, F)_{h, \omega}}.
$$
Finally, we require the following very famous H\"{o}rmander $L^2$ estimates, which will be used in the proof of our vanishing theorems. 
\begin{lem}
[{\cite{AV01,demailly,Hormander}}] \label{L2} Let $(X,\omega)$
be a complete K\"{a}hler manifold. Let $(E,h)$ be an hermitian
vector bundle over $X$. Assume that
$A=[i\Theta_{h}(E),\Lambda_\omega]$ is positive definite everywhere
on $\Lambda^{p,q}T^{*}X\otimes E$, $q\geq1$. Then for any form $g\in
L^2(X,\Lambda^{p,q}T^{*}X\otimes E)$ satisfying $\overline{\partial}
g=0$ and $\int_X (A^{-1}g,g)dV_\omega <+\infty,$
 there exists $f\in L^2(X,\Lambda^{p,q-1}T^{*}X\otimes E)$ such that $\overline{\partial}f=g$ and
 \begin{equation}
   \int_X|f|^2 dV_\omega\leq \int_X (A^{-1}g,g)dV_\omega.\nonumber
 \end{equation}
\end{lem}

\section{Proof of injectivity theorems}\label{poi.}
\begin{thm}[=Theorem \ref{chunle}]\label{chunle.}
Let $L$ be a semi-positive holomorphic line bundle over a compact K\"ahler manifold $X$ with a smooth hermitian metric $h_L$ satisfying $\sqrt{-1}\Theta_{h_L} (L) \geq 0$. If $F$ $($resp.~$E$$)$ is a holomorphic line $($resp.~vector$)$ bundle over $X$ with a smooth hermitian 
metric $h_F$ $($resp.~ $h_E$$)$ such that 
\begin{enumerate}
\item $\sqrt{-1}\Theta_{h_F}(F) - a \sqrt{-1}\Theta_{h_L} (L) \geq 0$
\item $\sqrt{-1}\Theta_{h_E}(E)+(a-b)Id_E\otimes\sqrt{-1}\Theta_{h_L}(L)\geq_{Nak}0$\label{eq7} \,\,\text{in the sense of Nakano} \end{enumerate}
for some positive constants $a, b>0$, then for a nonzero section $s \in H^{0}(X, L)$ the multiplication map induced by $\otimes s$ 
\begin{equation*}
\times s:  
H^{q}(X, K_{X}\otimes E \otimes F ) 
\xrightarrow{} 
H^{q}(X, K_{X} \otimes E\otimes F \otimes L )
\end{equation*}
is injective for every $q\geq0$, where $K_X$ is the canonical line bundle of $X$. 
\end{thm}
\begin{proof}
Let $\omega$ be a fixed \ka metric on $X$ and $n=\dim X$.
For simplicity, we denote $h_{E\otimes F}=h_E \otimes h_F$ and $h_{E\otimes F\otimes L}=h_E \otimes h_F\otimes h_L$.
By the Hodge theory it is enough to show the map
\begin{equation} \label{yj90}
  \times s: \mathcal{H}^{n,q}(X, E\otimes F)\rightarrow \mathcal{H}^{n,q}(X, E\otimes F\otimes L)
\end{equation}
is injective for every $q\geq0$, where $\mathcal{H}^{n,q}(X, E\otimes F)$ is space of
$E\otimes F$-valued forms $u$ such that $u$ is
harmonic with respect to the metrics $\omega$ and $h_{E\otimes F}$,
and the same for $\mathcal{H}^{n,q}(X, E\otimes F\otimes G)$.

We will show below the map $($\ref{yj90}$)$ is well-defined,
from which the injectivity is obvious to see. In fact, 
for any $u\in \mathcal{H}^{n,q}(X, E\otimes F)$
we have $\Delta''_{E\otimes F,h_{E\otimes F}}u=0$,
where $$\Delta''_{E\otimes F,h_{E\otimes F}}=
\nabla''_{E\otimes F}\nabla''^{*}_{E\otimes F,h_{E\otimes F}}+
\nabla''^{*}_{E\otimes F,h_{E\otimes F}}\nabla''_{E\otimes F}$$
is the complex Laplace-Beltrami operator of $\nabla''_{E\otimes F}$.
By the Nakano identity
\begin{equation}
\begin{split}
  &\|\nabla''_{E\otimes F}u\|_{\omega,h_{E\otimes F}}^2+
  \|\nabla_{E\otimes F,h_{E\otimes F}}^{''*}u\|_{\omega,h_{E\otimes F}}^2\\
  &=\|\nabla'^{*}_{E\otimes F}u\|_{\omega,h_{E\otimes F}}^2+
  \langle\langle  \sqrt{-1}\Theta_{h_{E\otimes F}}(E\otimes F)
  \Lambda_\omega u,u  \rangle\rangle_{\omega,h_{E\otimes F}}.\nonumber
\end{split}
\end{equation}
We note that $\Delta''_{E\otimes F,h_{E\otimes F}}u=0\Leftrightarrow
\nabla''_{E\otimes F}u=\nabla''^{*}_{E\otimes F,h_{E\otimes F}}u=0$.
Thus for any $u\in \mathcal{H}^{n,q}(X,E\otimes F)$ we have
\begin{equation}\label{711_1}
  0=\|\nabla'^{*}_{E\otimes F}u\|_{\omega,h_{E\otimes F}}^2+
  \langle\langle  \sqrt{-1}\Theta_{h_{E\otimes F}}(E\otimes F)
  \Lambda_\omega u,u  \rangle\rangle_{\omega,h_{E\otimes F}}.\end{equation}
But 
\begin{align*}
&\sqrt{-1}\Theta_{h_{E\otimes F}}(E\otimes F)\\
&=\sqrt{-1}\Theta_{h_E}(E)+Id_E\otimes\sqrt{-1}\Theta_{h_F}(F)\\
&\geq_{Nak}\sqrt{-1}\Theta_{h_E}(E)+Id_E\otimes a\sqrt{-1}\Theta_{h_L}(L)\\
&=\sqrt{-1}\Theta_{h_E}(E)+Id_E\otimes ((a-b)\sqrt{-1}\Theta_{h_L}(L)+b\sqrt{-1}\Theta_{h_L}(L))\\
&=\sqrt{-1}\Theta_{h_E}(E)+(a-b)Id_E\otimes \sqrt{-1}\Theta_{h_L}(L)+bId_E\otimes\sqrt{-1}\Theta_{h_L}(L)\\
&\geq_{Nak}0
\end{align*}
which means that $E\otimes F$ is Nakano semi-positive. It follows that 
the curvature operator $$[\sqrt{-1}\Theta_{h_{E\otimes F}}(E\otimes F),\Lambda_\omega]u=\sqrt{-1}\Theta_{h_{E\otimes F}}(E\otimes F)\Lambda_\omega u$$ is semi-positive for any $(n,q)$-forms $u$ on $X$.
Thus we obtain  
$$\|\nabla'^{*}_{E\otimes F}u\|_{\omega,h_{E\otimes F}}^2=
\langle\langle  \sqrt{-1}\Theta_{h_{E\otimes F}}(E\otimes F)
\Lambda_\omega u,u  \rangle\rangle_{\omega,h_{E\otimes F}}=0$$
by equation $($\ref{711_1}$)$.
It follows that
\begin{equation} \label{ikhg.1}
  \nabla'^{*}_{E\otimes F}u=\langle  \sqrt{-1}\Theta_{h_{E\otimes F}}(E\otimes F)\Lambda_\omega u,u  \rangle_{\omega,h_{E\otimes F}}=0
\end{equation}
where $\langle \bullet  \rangle_{\omega,h_{E\otimes F}}$
means the pointwise inner product on $X$ with respect to $\omega$ and $h_{E\otimes F}$.
Therefore 
\begin{equation}\label{7676.1}
  {\nabla'^{*}}_{E\otimes F\otimes L}(s u)=-*\nabla''_{E\otimes F\otimes L}*(su)=s\nabla'^{*}_{E\otimes F}u=0
\end{equation}
since $s$ is a holomorphic $L$-valued (0,0)-form,
where $*$ is the Hodge star operator with respect to the metric $\omega$.
By the Nakano identity again
\begin{equation}
\begin{split}
&\|\nabla''_{E\otimes F\otimes L}su\|_{\omega,h_{E\otimes F\otimes L}}^2+
  \|\nabla_{E\otimes F\otimes L,h_{E\otimes F\otimes L}}^{''*}su\|_{\omega,h_{E\otimes F\otimes L}}^2\\
&=\|\nabla'^{*}_{E\otimes F\otimes L}su\|_{\omega,h_{E\otimes F\otimes L}}^2+
  \langle\langle  \sqrt{-1}\Theta_{h_{E\otimes F\otimes L}}(E\otimes F\otimes L)
  \Lambda_\omega su,su  \rangle\rangle_{\omega,h_{E\otimes F\otimes L}}.\nonumber
\end{split}
\end{equation}
$\nabla''_{E\otimes F\otimes L}su=0$ by the Leibnitz rule,
since $s$ is holomorphic and $u$ is harmonic.
It follows that
\begin{equation}
 \|\nabla_{E\otimes F\otimes L,h_{E\otimes F\otimes L}}^{''*}su\|_{\omega,h_{E\otimes F\otimes L}}^2
 =\langle\langle  \sqrt{-1}\Theta_{h_{E\otimes F\otimes L}}(E\otimes F\otimes L)
  \Lambda_\omega su,su  \rangle\rangle_{\omega,h_{E\otimes F\otimes L}}.\nonumber
\end{equation}
On the other hand, we compute 
\begin{align*}
&\sqrt{-1}\Theta_{h_{E\otimes F}}(E\otimes F)\\
&=\sqrt{-1}\Theta_{h_E}(E)+Id_E\otimes\sqrt{-1}\Theta_{h_F}(F)\\
&\geq_{Nak}\sqrt{-1}\Theta_{h_E}(E)+Id_E\otimes a\sqrt{-1}\Theta_{h_L}(L)\\
&=\sqrt{-1}\Theta_{h_E}(E)+Id_E\otimes ((a-b)\sqrt{-1}\Theta_{h_L}(L)+b\sqrt{-1}\Theta_{h_L}(L))\\
&=\sqrt{-1}\Theta_{h_E}(E)+(a-b)Id_E\otimes \sqrt{-1}\Theta_{h_L}(L)+bId_E\otimes\sqrt{-1}\Theta_{h_L}(L)\\
&\geq_{Nak}bId_E\otimes\sqrt{-1}\Theta_{h_L}(L)
\end{align*}
that is, $$Id_E\otimes\sqrt{-1}\Theta_{h_L}(L)\leq_{Nak}\frac{1}{b}\sqrt{-1}\Theta_{h_{E\otimes F}}(E\otimes F).$$
It follows that  
\begin{align*}
&\sqrt{-1}\Theta_{h_{E\otimes F\otimes L}}(E\otimes F\otimes L)\\
&=\sqrt{-1}\Theta_{h_{E\otimes F}}(E\otimes F)+Id_E\otimes\sqrt{-1}\Theta_{h_{L}}(L)\\
&\leq_{Nak}(1+\frac{1}{b})\sqrt{-1}\Theta_{h_{E\otimes F}}(E\otimes F)
\end{align*}
Therefore, by equation $($\ref{ikhg.1}$)$ we have
\begin{align*}
&\langle \sqrt{-1}\Theta_{h_{E\otimes F\otimes L}}(E\otimes F\otimes L)\Lambda_\omega su,su  \rangle_{\omega,h_{E\otimes F\otimes L}}\\
&\leq  (1+\frac{1}{b})\langle\sqrt{-1}\Theta_{h_{E\otimes F}}(E\otimes F)\Lambda_\omega su,su  \rangle_{\omega,h_{E\otimes F\otimes L}}\\
&=(1+\frac{1}{b})|s|^2_{h_L}\langle \sqrt{-1}\Theta_{h_{E\otimes F}}(E\otimes F)\Lambda_\omega u,u  \rangle_{\omega,h_{E\otimes F}}\\
&=0 \nonumber
\end{align*}
and 
\begin{equation}
 \|\nabla_{E\otimes F\otimes L,h_{E\otimes F\otimes L}}^{''*}su\|_{\omega,h_{E\otimes F\otimes L}}^2
 =\langle\langle  \sqrt{-1}\Theta_{h_{E\otimes F\otimes L}}(E\otimes F\otimes L)
  \Lambda_\omega su,su  \rangle\rangle_{\omega,h_{E\otimes F\otimes L}}\leq0.\nonumber
\end{equation}
This means that 
$$\|\nabla_{E\otimes F\otimes L,h_{E\otimes F\otimes L}}^{''*}su\|_{\omega,h_{E\otimes F\otimes L}}^2=0\,\,\text{and}\,\,\nabla_{E\otimes F\otimes L,h_{E\otimes F\otimes L}}^{''*}su=0.$$
Recall that $\nabla''_{E\otimes F\otimes G}su=0$. Thus we conclude that 
the $E\otimes F\otimes L$-valued form $su$ is harmonic
with respect to the metrics $\omega$ and $h_{E\otimes F\otimes L}$, that is, $su\in \mathcal{H}^{n,q}(X, E\otimes F\otimes L).$ This means that the map (\ref{yj90}) is well-defined. The proof is finished.
\end{proof}
\begin{cor}[=Corollary \ref{en}]\label{en.}
Let $L$ be a semi-positive line bundle over a compact K\"ahler manifold $X$. Then for a nonzero section $s \in H^{0}(X, L^l)$ 
the multiplication map induced by $\otimes s$
\begin{equation}
  \times s: H^q(X,K_X\otimes L^k)\rightarrow H^q(X,K_X\otimes L^{l+k})\nonumber
\end{equation}
is injective for any $k,l\geq1$ and $q\geq0$.
\end{cor}
\begin{proof}
Let $E$ be the trivial line bundle on $X$. For the semi-positive line bundle $L$ we set $F=L^k$ and $L'=L^{l}$. Then the conditions (1) and (2) in Theorem \ref{chunle.} are easy to check for small positive constants $a>0,b>0$ with $a=b$. By Theorem \ref{chunle.} we know that 
for a nonzero section $s \in H^{0}(X, L')$ the multiplication map induced by $\otimes s$ 
\begin{equation*}
\times s:  
H^{q}(X, K_{X}\otimes E \otimes F ) 
\xrightarrow{} 
H^{q}(X, K_{X} \otimes E\otimes F \otimes L' )
\end{equation*}
that is, 
\begin{equation*}
\times s:  
H^{q}(X, K_{X}\otimes L^k ) 
\xrightarrow{} 
H^{q}(X, K_{X}\otimes L^{l+k} )
\end{equation*}
is injective for every $q\geq0$.
\end{proof}
\begin{cor}[=Corollary \ref{en_}]\label{en_.}
Let $L$ $($resp.~ $E$$)$ be a semi-positive line bundle $($resp.~ a Nakano semi-positive vector bundle$)$ over a compact K\"ahler manifold $X$. Then for a nonzero section $s \in H^{0}(X, L^l)$ 
the multiplication map induced by $\otimes s$
\begin{equation}
  \times s: H^q(X,K_X\otimes E\otimes L^k)\rightarrow H^q(X,K_X\otimes E \otimes L^{l+k})\nonumber
\end{equation}
is injective for any $k,l\geq1$ and $q\geq0$.
\end{cor}
\begin{proof}
For a semi-positive line bundle $L$ and a Nakano semi-positive vector bundle $E$ on $X$ we let $F=L^k$ and $L'=L^{l}$. Then the conditions (1) and (2) in Theorem \ref{chunle.} are easy to check for small positive constants $a>0,b>0$ with $a=b$. By Theorem \ref{chunle.} we know that 
for a nonzero section $s \in H^{0}(X, L')$ the multiplication map induced by $\otimes s$ 
\begin{equation*}
\times s:  
H^{q}(X, K_{X}\otimes E \otimes F ) 
\xrightarrow{} 
H^{q}(X, K_{X} \otimes E\otimes F \otimes L' )
\end{equation*}
that is, 
\begin{equation}
  \times s: H^q(X,K_X\otimes E\otimes L^k)\rightarrow H^q(X,K_X\otimes E \otimes L^{l+k})\nonumber
\end{equation}
is injective for any $k,l\geq1$ and $q\geq0$.
\end{proof}
\begin{thm}[=Theorem \ref{f-thm5.1}]\label{f-thm5.1.}
Let $L$ be a semi-positive holomorphic line bundle over a compact K\"ahler manifold $X$ with a smooth hermitian metric $h_L$ satisfying $\sqrt{-1}\Theta_{h_L} (L) \geq 0$.
If $F$ $($resp.~$E$$)$ is a holomorphic line $($resp.~vector$)$ bundle over $X$ with a singular hermitian 
metric $h$ $($resp.~a smooth hermitian metric $h_E$$)$ such that 
\begin{enumerate}
\item $\sqrt{-1}\Theta_{h}(F) - a \sqrt{-1}\Theta_{h_L} (L) \geq 0$ \,\,\text{in the sense of currents} 
\item $\sqrt{-1}\Theta_{h_E}(E)+(a-b)Id_E\otimes\sqrt{-1}\Theta_{h_L}(L)\geq_{Nak}0$\label{eq7} \,\,\text{in the sense of Nakano} \end{enumerate}
for some positive constants $a, b>0$, then for a nonzero section $s \in H^{0}(X, L)$ the multiplication map induced by $\otimes s$ 
\begin{equation*}
\times s:  
H^{q}(X, K_{X}\otimes E \otimes F \otimes \mathcal I(h)) 
\xrightarrow{} 
H^{q}(X, K_{X} \otimes E\otimes F \otimes \mathcal I(h) \otimes L )
\end{equation*}
is injective for every $q\geq0$, where 
$\mathcal I(h)$ is the multiplier ideal sheaf of $h$. 
\end{thm}
\begin{proof}
We may assume $q>0$ since the case $q=0$ is obvious.  
For the proof, it is sufficient to show that 
an arbitrary cohomology class 
$\eta \in H^{q}(X, K_{X}\otimes E\otimes F \otimes \I{h})$ 
satisfying $s \eta = 0 \in H^{q}(X, K_{X}\otimes E\otimes F \otimes \I{h} \otimes L)$ 
is actually zero. We fix a K\"ahler form  $\omega$ on $X$ throughout the proof and represent the cohomology 
class $\eta \in H^{q}(X, K_{X}\otimes E\otimes F \otimes \I{h})$ 
by a $\dbar$-closed $E\otimes F$-valued $(n,q)$-form $u$ 
with $\| u \|_{h_Eh, \omega} < \infty$ 
by using the standard De Rham--Weil isomorphism 
\begin{equation*}
H^{q}(X, K_{X}\otimes E\otimes F \otimes \I{h})
\cong 
\frac{\Ker\dbar: L^{n,q}_{(2)}(E\otimes F)_{h_Eh,\omega}
\to L^{n,q+1}_{(2)}(E\otimes F)_{h_Eh,\omega}}
{\Image\dbar: L^{n,q-1}_{(2)}(E\otimes F)_{h_Eh,\omega}
\to L^{n,q}_{(2)}(E\otimes F)_{h_Eh,\omega}}.
\end{equation*}
For the given singular hermitian metric $h$ on $F$, 
by the Demailly-Peternell-Schneider equisingular approximation theorem (Lemma \ref{equi.}), there is a countable family $\{h_{\e} \}_{1\gg \e>0}$ of singular hermitian metrics on $F$ with the following properties: 
\begin{itemize}
\item[(a)]$h_{\e}$ is smooth on $Y_{\e}:=X \setminus Z_{\e}$, 
where $Z_{\e}$ is a proper closed subvariety on $X$. 
\item[(b)]$h_{\e'} \leq h_{\e''} \leq h$ holds on $X$ 
when $\e' > \e'' > 0$.
\item[(c)]$\mathcal I({h})= \mathcal I({h_{\e}})$ on $X$.
\item[(d)]$\sqrt{-1} \Theta_{h_{\e}}(F) 
\geq a\sqrt{-1}\Theta_{h_{F}}(F) -\e \omega$ on $X$. 
\end{itemize}
By \cite[Section 3]{fujino-osaka} we can take a complete K\"ahler form  $\omega_{\e}$ on $Y_{\e}$ 
such that: $\omega_{\e}$ is a complete K\"ahler form on 
$Y_{\e}$, $\omega_{\e} \geq \omega $ on $Y_{\e}$ and $\omega_{\e}=\deldel \Psi_{\e}$ for some bounded function $\Psi_{\e}$ 
on a neighborhood of every $p \in X$.    
We define a K\"ahler form $\omega_{\e, \delta}$ on $Y_{\e}$ by 
$$
\omega_{\e, \delta}:=\omega + \delta \omega_{\e} 
$$
for $\e$ and $\delta$ with $0<\delta \ll \e$. 
The following properties are easy to check 
\begin{itemize}
\item[(A)] $\omega_{\e, \delta}$ is a 
complete K\"ahler form on $Y_{\e}=X\setminus Z_{\e}$ for every 
$\delta>0$.
\item[(B)] $\omega_{\e, \delta} \geq \omega $ on $Y_{\e}$ 
for every  $\delta>0$. 
\item[(C)] $\Psi + \delta \Psi_{\e}$ is a bounded local potential function of  $\omega_{\e, \delta}$ and converges to $\Psi$ as $\delta \to 0$ where 
$\Psi$ is a local potential function of $\omega$.  
\end{itemize} 

In the proof of Theorem \ref{f-thm5.1}, we actually consider only a countable sequence $\{\e_{k}\}_{k=1}^{\infty}$ 
(resp.~$\{\delta_{\ell}\}_{\ell=1}^{\infty}$) converging to zero since 
we need to apply Cantor's diagonal argument, 
but we often use the notation $\e$ (resp.~$\delta$) for simplicity. 
In the following, we mainly consider the $L^{2}$-space 
$L^{n,q}_{(2)}(Y_{\e},E\otimes F)_{h_Eh_{\e},\omega_{\e, \delta}}$ 
of $E\otimes F$-valued $(n,q)$-forms on $Y_{\e}$. 
We denote $L^{n,q}_{(2)}(E\otimes F)_{\e, \delta}:= L^{n,q}_{(2)}(Y_{\e},E\otimes F)_{h_Eh_{\e},\omega_{\e, \delta}}$ and 
$\|\bullet\|_{\e, \delta}:=\|\bullet\|_{h_Eh_{\e}, \omega_{\e, \delta}}$ 
for simplicity. The following inequality is easy to check
\begin{align}\label{eq5.1}
\|u\|_{\e, \delta} \leq \|u\|_{h_Eh, \omega_{\e, \delta}} \leq \|u\|_{h_Eh, \omega} <\infty.  
\end{align}
In particular, the norm $\|u\|_{\e, \delta}$ is uniformly bounded 
with respect to $\e$, $\delta$.

There are various formulations for $L^2$-estimates for $\dbar$-equations,
which originated from H\"{o}rmander's paper \cite{Hormander}.
The following one is suitable for our purpose. 
\begin{lem}[{cf.~\cite[4.1\,Th\'eor\`eme]{demailly-dbar}}]\label{f-lem5.12}
Assume that $B$ is a Stein open set in $X$ such that 
$\omega_{\e, \delta}=\deldel 
(\Psi + \delta \Psi_{\e})$ on a neighborhood of $\overline B$. 
Then, for an arbitrary 
$\alpha \in \Ker \dbar \subset L^{n, q}_{(2)}(B\setminus Z_{\e}, E\otimes F)_{\e, \delta}$, 
there exist $\beta \in L^{n, q-1}_{(2)}(B\setminus Z_{\e}, E\otimes F)_{\e, \delta}$ 
and a positive constant $C_{\e, \delta}$ $($independent of $\alpha$$)$ 
such that: $(1)$ $\dbar \beta = \alpha$  and 
$\|\beta \|^{2}_{\e, \delta}\leq C_{\e, \delta}  \|\alpha \|^{2}_{\e, \delta}$; $(2)$ $\varlimsup_{\delta \to 0} 
C_{\e,\delta}$ is finite and is independent of $\e$.
\end{lem}
\begin{proof}[Proof of Lemma \ref{f-lem5.12}]
We may assume $\e < 1/2$. 
For the smooth hermitian metric $H_{\e,\delta}$ on $E\otimes F$ over 
$B\setminus Z_{\e}$ defined by 
$H_{\e,\delta}:=h_Eh_{\e} e^{-(\Psi + \delta \Psi_{\e})}$, 
the curvature satisfies $$\sqrt{-1}\Theta_{H_{\e,\delta}}(E\otimes F)\geq_{Nak} {1}/{2}\cdot Id_E\otimes\omega_{\e,\delta}$$
by property (B) and 
$\sqrt{-1}\Theta_{h_Eh_{\e}}(E\otimes F) \geq -\e Id_E\otimes \omega$.  
The $L^{2}$-norm $\|\alpha \|_{H_{\e,\delta}, 
\omega_{\e,\delta}}$ with respect to $H_{\e,\delta}$ 
and $\omega_{\e,\delta}$ is finite 
since the function $\Psi + \delta \Psi_{\e} $ is bounded and 
$\|\alpha \|_{\e,\delta}$ is finite. 
Therefore, from the standard $L^{2}$-method for the $\dbar$-equation 
(cf. \cite[4.1\,Th\'eor\`eme]{demailly-dbar}), 
we obtain a solution $\beta$ of the $\dbar$-equation 
$\dbar \beta =\alpha$ with  $$\|\beta \|^{2}_{H_{\e,\delta},\omega_{\e,\delta}} 
\leq \frac{2}{q} \|\alpha \|^{2}_{H_{\e,\delta},\omega_{\e,\delta}}.  
$$ It follows that 
$$\|\beta \|^2_{\e, \delta}  \leq 
C_{\e, \delta}
\|\alpha \|^2_{\e, \delta}$$ 
where $C_{\e, \delta}=\frac{2}{q}
\frac{ \sup_{B} e^{-(\Psi + \delta \Psi_{\e})} }
{\inf_{B} e^{-(\Psi + \delta \Psi_{\e})}}$. 
It is easy to check $C_{\e, \delta}$ satisfies the above properties.
\end{proof}
By essentially using the property (C) and Lemma \ref{f-lem5.12} we have the following De Rham--Weil isomorphism 
from the $\dbar$-$L^2$ cohomology on $Y_{\e}$ to 
the $\rm{\check{C}}$ech cohomology on $X$ 
(cf. \cite[Claim 1]{fujino-osaka} and \cite[Proposition 5.5]{matsumura4}) 
\begin{equation*}\label{FujMatsu}
H^{q}(X, K_{X}\otimes E\otimes F \otimes \I{h_{\e}})
\cong 
\frac{\Ker\dbar: L^{n,q}_{(2)}(E\otimes F)_{\e, \delta}
\to L^{n,q+1}_{(2)}(E\otimes F)_{\e, \delta}}
{\Image\dbar: L^{n,q-1}_{(2)}(E\otimes F)_{\e, \delta}
\to L^{n,q}_{(2)}(E\otimes F)_{\e, \delta}}. 
\end{equation*}
The following orthogonal decomposition then follows 
$$L^{n,q}_{(2)}(E\otimes F)_{\e, \delta}= 
\Image\dbar \,  \oplus 
\mathcal{H}^{n,q}_{\e, \delta}(E\otimes F)\, \oplus 
\Image \dbar^{*}_{\e, \delta}.$$
Now that the $E\otimes F$-valued $(n,q)$-form $u$  
belongs to $L^{n,q}_{(2)}(E\otimes F)_{\e,\delta}$ by \eqref{eq5.1}, 
it can be decomposed into $u=\dbar w_{\e, \delta}+u_{\e, \delta}$ for some $w_{\e, \delta} \in {\rm{Dom}\, \dbar} \subset 
L^{n,q-1}_{(2)}(E\otimes F)_{\e,\delta}\text{ and } 
u_{\e, \delta} \in \mathcal{H}^{n,q}_{\e, \delta}(E\otimes F). $
The orthogonal projection of $u$ to 
${\Image \dbar^{*}_{\e, \delta}}$ 
is zero since $u$ is $\dbar$-closed.
We need the following Lemma \ref{f-prop5.7} which can be proved by the same analytic method provided by O. Fujino and S. Matsumura in \cite{FM,matsumura4}. Here we omit the proof for simplicity. For the details, we refer the reader to the proof of Proposition 5.7 in \cite{FM} in which the inequality (\ref{eq5.1}) plays an important role. By Lemma \ref{f-prop5.7} it is sufficient for the proof to study the asymptotic behavior of the norm of $su_{\e, \delta}$.
\begin{lem}[cf. Proposition 5.7 in \cite{FM}]\label{f-prop5.7}
The cohomology class $\eta$ is zero if 
$$
\varliminf_{\e \to 0} \varliminf_{\delta \to 0} \|su_{\e, \delta}\|_{\e, \delta} =0
$$
where $\|\bullet\|_{\e, \delta}:= 
\|\bullet\|_{h_Eh_{\e}h_L, \omega_{\e,\delta}}$
for an $E\otimes F \otimes L$-valued form $\bullet$.
\end{lem}

Following the proof of Proposition 5.8 in \cite{FM} 
and Proposition 2.8 in \cite{matsumura1} we have 
\begin{lem}\label{f-prop5.8}
$$\lim_{\e \to 0} \varlimsup_{\delta \to 0}
\| \dbar^{*}_{\e, \delta} s u_{\e, \delta} 
\|_{\e, \delta}=0.$$
\end{lem}
\begin{proof}[Proof of Lemma \ref{f-prop5.8}]
By applying the Bochner--Kodaira--Nakano identity and the density lemma to $u_{\e, \delta}$ and $s u_{\e, \delta}$ 
(see \cite[Proposition 5.8]{FM} and \cite[Proposition 2.8]{matsumura1}), 
we have 
\begin{align}
0 &= \lla \sqrt{-1}\Theta_{h_Eh_{\e}}(E\otimes F)
\Lambda_{\omega_{\e, \delta}} u_{\e,\delta}, u_{\e,\delta}
  \rra_{\e,\delta} + \|D'^{*}_{\e,\delta}u_{\e,\delta} \|^{2}_{\e,\delta}\label{eq5.4}\\
\| \dbar^{*}_{\e, \delta} s u_{\e, \delta} 
\|^2_{\e, \delta} &= \lla \sqrt{-1}\Theta_{h_Eh_{\e}h_{L}}(E\otimes F\otimes L)
\Lambda_{\omega_{\e, \delta}} su_{\e, \delta}, su_{\e, \delta}
  \rra_{\e, \delta} + \|D'^{*}_{\e, \delta}su_{\e, \delta} \|^{2}_{\e, \delta}    \label{eq5.5}
\end{align}
where we used the fact that $u_{\e, \delta}$ is harmonic 
and $\dbar (s u_{\e,\delta})=s\dbar u_{\e,\delta}=0$. 
We have 
\begin{equation}\label{huaixi}
\begin{split}
&\sqrt{-1}\Theta_{h_Eh_{\e}}(E\otimes F) 
=\sqrt{-1}\Theta_{h_E}(E) +Id_E\otimes \sqrt{-1}\Theta_{h_{\e}}(F) \\
&\geq_{Nak} \sqrt{-1}\Theta_{h_E}(E) +Id_E\otimes (a \sqrt{-1}\Theta_{h_L}(L)  -\e \omega )\\
&=\sqrt{-1}\Theta_{h_E}(E) +Id_E\otimes ((a-b) \sqrt{-1}\Theta_{h_L}(L)  +b\sqrt{-1}\Theta_{h_L}(L) -\e \omega )\\
&\geq_{Nak}Id_E\otimes (b\sqrt{-1}\Theta_{h_L}(L) -\e \omega )
\end{split}
\end{equation}
by properties (d), (B) and the assumption (\ref{eq7}) in 
Theorem \ref{f-thm5.1}. It follows that 
\begin{align*}
\sqrt{-1}\Theta_{h_Eh_{\e}}(E\otimes F)\geq_{Nak} -\e Id_E\otimes \omega \geq_{Nak} -\e Id_E\otimes \omega_{\e,\delta}.
\end{align*}
So the integrand $g_{\e, \delta}$ of the first term 
of \eqref{eq5.4} satisfies 
\begin{equation}\label{eq5.6}
 -\e q |u_{\e, \delta}|^{2}_{\e, \delta} 
 \leq g_{\e, \delta}:= \langle \sqrt{-1}\Theta_{h_Eh_{\e}}(E\otimes F)
\Lambda_{\omega_{\e, \delta}} u_{\e,\delta}, u_{\e,\delta}
  \rangle_{\e,\delta}. 
\end{equation}
For the precise 
argument, see \cite[Step 2 in the proof of Theorem 3.1]{matsumura4}. 
By \eqref{eq5.4} we have 
\begin{align*}
&\lim_{\e \to 0} \varlimsup_{\delta \to 0} 
\Big( \int_{\{g_{\e, \delta} \geq 0\}} g_{\e, \delta}\,dV_{\omega_{\e, \delta}}
+\|D'^{*}_{\e,\delta}u_{\e,\delta} \|^{2}_{\e,\delta} \Big)\\
&= \lim_{\e \to 0} \varlimsup_{\delta \to 0} 
\Big(-\int_{\{g_{\e, \delta} \leq 0\}} g_{\e, \delta}\, dV_{\omega_{\e, \delta}}\Big)\\
&\leq \lim_{\e \to 0} \varlimsup_{\delta \to 0} 
\Big(\e q \int_{\{g_{\e, \delta} \leq 0\}}  |
u_{\e, \delta}|^{2}_{\e, \delta}\, dV_{\omega_{\e, \delta}}\Big)\\
&\leq  \lim_{\e \to 0} \varlimsup_{\delta \to 0} 
\Big(\e q \|u_{\e, \delta} \|^2_{\e, \delta}\Big)=0.
\end{align*}
It follows that 
\begin{equation}\label{A4}
\lim_{\e \to 0} \varlimsup_{\delta \to 0} 
\int_{\{g_{\e, \delta} \geq 0\}} g_{\e, \delta}\,dV_{\omega_{\e, \delta}}=0
\,\,\text{and}\,\,
\lim_{\e \to 0} \varlimsup_{\delta \to 0} 
\|D'^{*}_{\e,\delta}u_{\e,\delta} \|^{2}_{\e,\delta}=0.
\end{equation}
Therefore, by equation (\ref{eq5.4}) we have 
\begin{equation}
\lim_{\e \to 0} \varlimsup_{\delta \to 0} 
\lla \sqrt{-1}\Theta_{h_Eh_{\e}}(E\otimes F)
\Lambda_{\omega_{\e, \delta}} u_{\e,\delta}, u_{\e,\delta}
  \rra_{\e,\delta}=0
\end{equation}
Thus, we obtain
\begin{equation}\label{chunle7171}
0\leq\lim_{\e \to 0} \varlimsup_{\delta \to 0} 
\lla \sqrt{-1}\Theta_{h_Eh_{\e}h_{L}}(E\otimes F\otimes L)
\Lambda_{\omega_{\e, \delta}} su_{\e, \delta}, su_{\e, \delta} \rra_{\e, \delta}
\end{equation}
thanks to
\begin{equation*}
\sqrt{-1}\Theta_{h_Eh_{\e}}(E\otimes F)\leq_{Nak}\sqrt{-1}\Theta_{h_Eh_{\e}h_{L}}(E\otimes F\otimes L)
\end{equation*}
and 
\begin{equation*}
\langle \sqrt{-1}\Theta_{h_Eh_{\e}}(E\otimes F)
\Lambda_{\omega_{\e, \delta}} su_{\e, \delta}, su_{\e, \delta} \rangle_{\e, \delta}\leq\langle \sqrt{-1}\Theta_{h_Eh_{\e}h_{L}}(E\otimes F\otimes L)
\Lambda_{\omega_{\e, \delta}} su_{\e, \delta}, su_{\e, \delta} \rangle_{\e, \delta}.
\end{equation*}
On the other hand, by formula (\ref{huaixi}) we have
$$\sqrt{-1}\Theta_{h_Eh_{\e}}(E\otimes F) \geq_{Nak}Id_E\otimes (b\sqrt{-1}\Theta_{h_L}(L) -\e \omega )
\geq_{Nak} Id_E\otimes (b\sqrt{-1}\Theta_{h_L}(L) -\e \omega_{\e, \delta}).$$
It follows that 
\begin{align*}
&\lla \sqrt{-1}\Theta_{h_Eh_{\e}h_{L}}(E\otimes F\otimes L)
\Lambda_{\omega_{\e, \delta}} su_{\e, \delta}, su_{\e, \delta} \rra_{\e, \delta}\\ 
\leq& \big(1+\frac{1}{b}\big)
\int_{Y_{\e}} |s|^2_{h_L}g_{\e,\delta}\,dV_{\omega_{\e, \delta}}
+\frac{\e q}{b} \int_{Y_{\e}} 
|s|^2_{h_L}|u_{\e, \delta}|^{2}_{\e, \delta} \,dV_{\omega_{\e, \delta}}\\
\leq& \big(1+\frac{1}{b}\big)\sup_{X} 
|s|^2_{h_L}\int_{\{g_{\e, \delta} \geq 0\}} g_{\e,\delta}\,dV_{\omega_{\e, \delta}}+
\frac{\e q}{b} \sup_{X} |s|^2_{h_L} \|u_{\e,\delta}\|^2_{\e,\delta}\\
\leq& \big(1+\frac{1}{b}\big)\sup_{X} 
|s|^2_{h_L}\int_{\{g_{\e, \delta} \geq 0\}} g_{\e,\delta}\,dV_{\omega_{\e, \delta}}+
\frac{\e q}{b} \sup_{X} |s|^2_{h_L} \|u\|^2_{h_Eh_{\e},\omega}
\end{align*}
which leads to 
\begin{equation}\label{eq4}
\lim_{\e \to 0} \varlimsup_{\delta \to 0} \lla\sqrt{-1}\Theta_{h_Eh_{\e}h_{L}}(E\otimes F\otimes L)\Lambda_{\omega_{\e, \delta}} su_{\e, \delta}, su_{\e, \delta}\rra_{\e, \delta}=0
\end{equation}
by formulas (\ref{A4}) and (\ref{chunle7171}). 
Moreover, we have 
$$\|D'^{*}_{\e, \delta} su_{\e, \delta} \|^{2}_{\e, \delta}
=\|s D'^{*}_{\e, \delta} u_{\e, \delta} \|^{2}_{\e, \delta}
\leq \sup_{X} |s|^2_{h_L} \|D'^{*}_{\e, \delta} u_{\e, \delta} \|^{2}_{\e, \delta}$$
thanks to $D'^{*}_{\e, \delta}=-*\dbar*$ where $*$ is the Hodge star operator with respect to $\omega_{\e,\delta}$. By formula (\ref{A4})
it follows that
\begin{equation}\label{eq5}
\lim_{\e \to 0} \varlimsup_{\delta \to 0} \|D'^{*}_{\e, \delta} su_{\e, \delta} \|^{2}_{\e, \delta}=0.
\end{equation}
Therefore, we obtain the conclusion 
by the equations (\ref{eq5.5}), (\ref{eq4}) and (\ref{eq5}).
\end{proof}

\begin{lem}\label{f-prop5.10}
There exist $E\otimes F\otimes L$-valued $(n, q-1)$-forms  
$v_{\e,\delta}$ on $Y_{\e}$ such that 
$\dbar v_{\e,\delta}=su_{\e, \delta}$ and 
$\varlimsup_{\delta \to 0} \| v_{\e,\delta}\|_{\e, \delta} $ 
can be bounded by a constant independent of $\e$. 
\end{lem}
The proof of Lemma \ref{f-prop5.10} is completely the same as that in Proposition 5.10 in \cite{FM} in which  Lemma \ref{f-lem5.12} is used to establish the De Rham--Weil isomorphism from the $\dbar$-$L^2$ cohomology on $Y_{\e}$ to the $\rm{\check{C}}$ech cohomology on $X$ (cf. \cite[Claim 1]{fujino-osaka} and \cite[Proposition 5.5]{matsumura4}) 
and the inequality (\ref{eq5.1}) is essentially used to control the bound. 
For the details, we refer the reader to the proof of Proposition 5.10 in \cite{FM} and here we omit it for simplicity. 
\begin{lem}\label{f-prop5.11}
\begin{align*}
\lim_{\e \to 0} \varlimsup_{\delta \to 0} \|s u_{\e, \delta}\|_{\e, \delta}=0. 
\end{align*}
\end{lem}
\begin{proof}[Proof of Lemma \ref{f-prop5.11}]
For the solution 
$v_{\e,\delta}$ in Lemma \ref{f-prop5.10}, 
it is easy to check that 
\begin{align*}
\lim_{\e \to 0} \varlimsup_{\delta \to 0} \|s u_{\e, \delta}\|^2_{\e, \delta}
=\lim_{\e \to 0} \varlimsup_{\delta \to 0} 
\lla \dbar^{*}_{\e,\delta} s u_{\e, \delta}, v_{\e,\delta}\rra_{\e, \delta}
\leq \lim_{\e \to 0} \varlimsup_{\delta \to 0}
\| \dbar^{*}_{\e,\delta} s u_{\e, \delta}\|_{\e,\delta} \|v_{\e,\delta}\|_{\e, \delta}. 
\end{align*}
By Lemma \ref{f-prop5.8} and Lemma \ref{f-prop5.10} we conclude that 
the right hand side is zero. 
\end{proof}
Now we finish the proof of Theorem \ref{f-thm5.1.} 
by Lemma \ref{f-prop5.7}. 
\end{proof}
\begin{cor}[=Corollary \ref{f-m}]\label{f-m.}
Let $L$ be a semi-positive holomorphic line bundle over a compact K\"ahler manifold $X$ with a smooth hermitian metric $h_L$ satisfying $\sqrt{-1}\Theta_{h_L} (L) \geq 0$.
If $F$ is a holomorphic line bundle over $X$ with a singular hermitian 
metric $h$ such that $\sqrt{-1}\Theta_{h}(F) - a \sqrt{-1}\Theta_{h_L} (L) \geq 0$ in the sense of currents for some positive constants $a>0$, then for a nonzero section $s \in H^{0}(X, L)$ the multiplication map induced by $\otimes s$ 
\begin{equation*}
\times s: 
H^{q}(X, K_{X}\otimes F \otimes \mathcal I(h)) 
\xrightarrow{}
H^{q}(X, K_{X}\otimes F \otimes \mathcal I(h) \otimes L )
\end{equation*}
is injective for every $q\geq0$.  
\end{cor}
\begin{proof}
We let $E$ be the trivial line bundle on $X$.
Then the conditions (1) and (2) in Theorem \ref{f-thm5.1.} are easy to check if we take $b=a$. By Theorem \ref{f-thm5.1.} we know that 
for a nonzero section $s \in H^{0}(X, L)$ the multiplication map induced by $\otimes s$ 
\begin{equation*}
\times s:  
H^{q}(X, K_{X}\otimes E \otimes F \otimes \mathcal I(h)) 
\xrightarrow{} 
H^{q}(X, K_{X} \otimes E\otimes F \otimes \mathcal I(h) \otimes L )
\end{equation*}
that is, 
\begin{equation*}
\times s:  
H^{q}(X, K_{X}\otimes F \otimes \mathcal I(h)) 
\xrightarrow{} 
H^{q}(X, K_{X}\otimes F \otimes \mathcal I(h) \otimes L )
\end{equation*}
is injective for every $q\geq0$, where 
$\mathcal I(h)$ is the multiplier ideal sheaf of $h$.
\end{proof}
\begin{cor}[=Corollary \ref{f-m2}]\label{f-m2.}
Let $L$ be a semi-positive holomorphic line bundle over a compact K\"ahler manifold $(X,\omega)$ equipped with a smooth hermitian metric $h_L$ satisfying $\sqrt{-1}\Theta_{h_L} (L) \geq 0$ and $E$ a Nakano semi-positive vector bundle over $X$.
If $F$ is a holomorphic line bundle over $X$ with a singular hermitian 
metric $h$ such that $\sqrt{-1}\Theta_{h}(F) - a \sqrt{-1}\Theta_{h_L} (L) \geq 0$ for some positive constants $a>0$, then for a nonzero section $s \in H^{0}(X, L)$ the multiplication map induced by $\otimes s$ 
\begin{equation*}
\times s:
H^{q}(X, K_{X}\otimes E \otimes F \otimes \mathcal I(h)) 
\xrightarrow{} 
H^{q}(X, K_{X} \otimes E\otimes F \otimes \mathcal I(h) \otimes L )
\end{equation*}
is injective for every $q\geq0$.  
\end{cor}
\begin{proof}
For a Nakano semi-positive vector bundle $E$ on $X$ 
the conditions (1) and (2) in Theorem \ref{f-thm5.1.} are easy to check if we take $b=a$. By Theorem \ref{f-thm5.1.} we know that 
for a nonzero section $s \in H^{0}(X, L)$ the multiplication map induced by $\otimes s$ 
\begin{equation*}
\times s:  
H^{q}(X, K_{X}\otimes E \otimes F \otimes \mathcal I(h)) 
\xrightarrow{} 
H^{q}(X, K_{X} \otimes E\otimes F \otimes \mathcal I(h) \otimes L )
\end{equation*}
is injective for every $q\geq0$, where 
$\mathcal I(h)$ is the multiplier ideal sheaf of $h$.
\end{proof}

\section{Proof of vanishing theorems}\label{pov.}
\begin{cor}[=Corollary \ref{f-thm1.3}]\label{f-thm1.3.}
Let $X$ be a smooth projective manifold with 
a K\"ahler form $\omega$ and $E$ a Nakano semi-positive vector bundle on $X$.
Let $F$ be a holomorphic line bundle on $X$ with 
a singular hermitian metric $h$ such that 
$\sqrt{-1}\Theta_{h}(F)\geq \varepsilon \omega$ in the sense of currents for some $\varepsilon >0$. Then for every $q>0$ we have 
$$H^q(X, K_X\otimes E\otimes F\otimes \mathcal I(h))=0.$$ 
\end{cor}
\begin{proof}
Let $A$ be an ample line bundle on $X$. 
Then there exists a sufficiently large positive integer $m$ such that 
$L:=A^{\otimes m}$ is very ample and that 
$$H^q(X, K_X\otimes E\otimes F\otimes \mathcal I(h)\otimes 
L)=0$$ for every $q>0$ by the Serre vanishing 
theorem. We can take a smooth hermitian metric $h_L$ on $L$ such that 
$\sqrt{-1}\Theta_{h_L}(L)\geq 0$. 
By the condition $\sqrt{-1}\Theta_{h}(F)\geq \varepsilon \omega$, 
we have $$\sqrt{-1}(\Theta_{h}(F)-a\Theta_{h_L}
(L))\geq 0$$ for some $0<a\ll 1$. 
On the other hand, we can take a smooth hermitian metric $h_E$ on $E$ such that $\sqrt{-1}\Theta_{h_E}(E)\geq_{Nak}0$ since $E$ is Nakano semi-positive by the assumption. It follows that 
$$\sqrt{-1}\Theta_{h_E}(E)+(a-b)Id_E\otimes\sqrt{-1}\Theta_{h_L}(L)\geq_{Nak}0$$
for some $0<b<a$.
We take a nonzero section $s \in H^{0}(X, L)$. By Theorem \ref{f-thm5.1.}, the multiplication map induced by $\otimes s$ 
\begin{equation*}
\times s: H^q(X, K_X\otimes E\otimes F\otimes \mathcal I(h))
\to H^q(X, K_X\otimes E\otimes F\otimes \mathcal I(h)\otimes 
L)
\end{equation*}
is injective for every $q>0$. 
Thus we obtain that $H^q(X, K_X\otimes E\otimes F\otimes 
\mathcal I(h))=0$ for every $q>0$. 
\end{proof}

\begin{cor}[=Corollary \ref{f-thm1.4}]\label{f-thm1.4.}
Let $X$ be a smooth projective manifold with 
a K\"ahler form $\omega$ and $F$ be a holomorphic line bundle on $X$ with a singular hermitian metric $h$ such that 
$\sqrt{-1}\Theta_{h}(F)\geq \varepsilon \omega$ in the sense of currents for some $\varepsilon >0$. Then for every $q>0$ we have 
$$
H^q(X, K_X\otimes F\otimes \mathcal I(h))=0.
$$ 
\end{cor}
\begin{proof}
It is obvious that Corollary \ref{f-thm1.4.} follows from Corollary \ref{f-thm1.3.} if we let $E$ be the trivial line bundle in Corollary \ref{f-thm1.3.}.
\end{proof}

\begin{thm} [=Theorem \ref{ND}]\label{ND.}
Let $(X,\omega)$ be a compact \ka manifold of dimension $n$,
$m$ a positive integer, $E$ a holomorphic vector bundle on $X$ of rank $r$ and $F$ a pseudo-effective line bundle on $X$ equipped with a singular hermitian metric $h$ with semi-positive curvature current.
If $E$ is Demailly $m$-positive then for any $q\geq 1$ with $m\geq\min\{n-q+1,r\}$ we have
\begin{equation}
H^q(X,K_X\otimes E\otimes F\otimes \mathcal{I}(h))=0 \nonumber
\end{equation}
where $K_X$ is the canonical bundle of $X$ and $\mathcal{I}(h)$
is the multiplier ideal sheaf of $h$.
\end{thm}
\begin{proof}
We fix a smooth Hermitian metric $h_E$ on $E$ such that the Chern curvature $\sqrt{-1}\Theta_{h_E}(E)$ is Demailly $m$-positive.
For the given singular hermitian
metric $h$ on $F$ with semi-positive curvature current,
by the Demailly-Peternell-Schneider equisingular approximation theorem (Lemma \ref{equi.}), there exist singular hermitian metrics $\{h_{\e} \}_{1\gg \e>0}$ on $F$ with the following properties\,$:$
\begin{itemize}
\item[(a)] $h_{\e}$ is smooth on $Y_\varepsilon:=X \setminus Z_{\e}$, where $Z_{\e}$ is a subvariety on $X$.
\item[(b)] $h_{\varepsilon_2}\leq h_{\varepsilon_1} \leq h_F$ for any $0<\varepsilon_1<\varepsilon_2$.
\item[(c)]$\I{h}= \I{h_{\e}}$  (equisingularity).
\item[(d)]$\sqrt{-1} \Theta_{h_{\e}}(F) \geq -\e \omega$.
\end{itemize}
It follows from the property (a) that the hermitian metric $h_E\otimes h_{\varepsilon}$ on $(E\otimes F)|_{Y_{\varepsilon}}$ is smooth on $Y_{\e}$ for every $h_{\varepsilon}\in \{h_{\e} \}_{1\gg \e>0}$. Further, 
it follows from the property (d) and the compactness of $X$ that
there exists a singular hermitian metric $h_{\varepsilon_0}\in \{h_{\e} \}_{1\gg \e>0}$
on $F$ such that the Chern curvature
$$\sqrt{-1}\Theta_{h_E\otimes h_{\varepsilon_0}}(E\otimes F)
  =\sqrt{-1}\Theta_{h_E}(E)+Id_E\otimes\sqrt{-1}\Theta_{h_{\varepsilon_0}}(F)$$
is Demailly $m$-positive on $Y_{\varepsilon_0}$. 
We denote $Y:=Y_{\varepsilon_0}$ and $Z:=Z_{\varepsilon_0}$ for simplicity. We need the following lemma which is more or less known to
experts(cf. \cite[Lemma 5]{Demailly03}).
\begin{lem} \cite[Lemma 5]{Demailly03}\label{lemma1}
There exists a quasi-psh function $\psi$ on $X$ such that
$\psi=-\infty$ on $Z$ with logarithmic poles along $Z$ and $\psi$
is smooth outside $Z$.
\end{lem}
\noindent By Lemma \ref{lemma1} we know that
there exists a quasi-psh function $\psi$ on $X$ such that $\psi=-\infty$ on $Z$,
with logarithmic poles along $Z$, and is smooth outside $Z$ .
Without loss of generality, we can assume that  $\psi<-e$ on $X$
since the quasi-psh function $\psi$ is upper semicontinuous and bounded above.
We define $$\varphi=1/\log(-\psi).$$ Then it is easy to check that
$\varphi$ is a quasi-psh function on $X$ and $0<\varphi<1$.
Thus, we can take a positive constant $\alpha$
such that $$\sqrt{-1}\partial\overline{\partial}\varphi+\alpha\omega>0$$ on $Y=X-Z.$
It follows from the Hopf--Rinow lemma that
\begin{equation}\label{Fuj.}
\widetilde{\omega}=\omega+(\sqrt{-1}\partial\overline{\partial}\varphi+\alpha\omega)
\end{equation}
will be a complete K\"ahler form on $Y$ if we choose $\alpha \gg0$.
For the details we recommend the reader to see \cite[Section 3]{fujino-osaka}. We denote by $L^{n,q}_{(2)}(Y,E\otimes F)_{\widetilde{\omega},h_E\otimes h_{\varepsilon_0}}$ 
the space of square integrable $E\otimes F$-valued $(n,q)$-forms on $Y$
with respect to the metrics $\widetilde{\omega}$ and $h_E\otimes h_{\varepsilon_0}$ and by 
$H^{n,q}_{(2)}(Y,E\otimes F)_{\widetilde{\omega},h_E\otimes h_{\varepsilon_0}}$ the $\dbar$-$L^2$ cohomology of
$L^{n,q}_{(2)}(Y,E\otimes F)_{\widetilde{\omega},h_E\otimes h_{\varepsilon_0}}$. 
The important point for our purpose is that the \ka metric $\widetilde{\omega}$ locally admits a bounded potential on a neighborhood of every point $p \in X$ (not $Y$). This means that for every point $p$ in $X$, there exist an open neighborhood $U$ of $p$ and a bounded function $\Psi$ on $U$ such that ${\omega}=\sqrt{-1}\partial\bar\partial\Psi$ on $U\setminus Z$ (see \cite{fujino-osaka,matsumura4}). 
By using this special property and the classical $L^2$-estimates for $\bar\partial$-equations on complete K\"ahler manifolds we can prove 
the following Lemma \ref{cla1} which plays an important role in the proof of Theorem \ref{ND}. For the details of the proof of Lemma \ref{cla1} see \cite{fujino-osaka,matsumura4}.
\begin{lem}\cite[Claim 1]{fujino-osaka}\label{cla1}
Under the same notation as above,
we have the following De Rham-Weil isomorphism between
the $\check{\rm{C}}$ech cohomology and the $\dbar$-$L^2$ cohomology
\begin{equation}\label{iso..}
  \check{H}^{q}(X, K_X\otimes E\otimes F\otimes \mathcal I(h))\simeq
H^{n,q}_{(2)}(Y,E\otimes F)_{\widetilde{\omega},h_E\otimes h_{\varepsilon_0}}. \nonumber
\end{equation}
\end{lem}
\noindent By Lemma \ref{cla1} it follows that
\begin{equation}\label{iso..}
H^{n,q}_{(2)}(Y,E\otimes F)_{\widetilde{\omega},h_E\otimes h_{\varepsilon_0}}
\simeq \mathbb{H}^{n,q}_{(2)}(Y,E\otimes F)_{\widetilde{\omega},h_E\otimes h_{\varepsilon_0}} \nonumber
\end{equation}
since the dimension of $H^{n,q}_{(2)}(Y,E\otimes F)_{\widetilde{\omega},h_E\otimes h_{\varepsilon_0}}$
is finite, where $\mathbb{H}^{n,q}_{(2)}(Y,E\otimes F)_{\widetilde{\omega},h_E\otimes h_{\varepsilon_0}}$
is the corresponding harmonic space of
$L^{n,q}_{(2)}(Y,E\otimes F)_{\widetilde{\omega},h_E\otimes h_{\varepsilon_0}}$. Now we see that: to prove Theorem \ref{ND} 
it's enough to show the vanishing
$$\mathbb{H}^{n,q}_{(2)}(Y,E\otimes F)_{\widetilde{\omega},h_E\otimes h_{\varepsilon_0}}=0$$
of the harmonic space for any $q\geq 1$ with $m\geq\min\{n-q+1,r\}$.

In fact, since the hermitian metric $h_E\otimes h_{\varepsilon}$
on $(E\otimes F)|_{Y}$ is smooth and
the Chern curvature $\sqrt{-1}\Theta_{h_E\otimes h_{\varepsilon_0}}(E\otimes F)$
is Demailly $m$-positive on $Y$ the curvature operator
$[\sqrt{-1}\Theta_{h_E\otimes h_{\varepsilon_0}}(E\otimes F),\Lambda_{\widetilde{\omega}}]$
on the vector bundle $\Lambda^{n,q}T_Y^{*}\otimes (E\otimes F)\rightarrow Y$ is positive
for any $q\geq 1$ with $m\geq\min\{n-q+1,r\}$
(cf. Lemma (7.2) in \cite{demailly-note}, page 341). This implies that
\begin{equation} \label{4242}
 \int_{Y}([\sqrt{-1}\Theta_{h_E\otimes h_{\varepsilon_0}}(E\otimes F),
 \Lambda_{\widetilde{\omega}}]u,u)dV_{\widetilde{\omega}}\leq
 \|\overline{\partial}u\|^2+\|\overline{\partial}^{*}u\|^2
\end{equation}
for any $u\in$Dom$\overline{\partial}\cap$ Dom $\overline{\partial}^{*}$
(cf. \cite{demailly-note}, p. 370).

Now if $u$ is an element in
$\mathbb{H}^{n,q}_{(2)}(Y,E\otimes F)_{\widetilde{\omega},h_E\otimes h_{\varepsilon_0}}$,
that is, $u$ 
is an element in $L^{n,q}_{(2)}(Y,E\otimes F)_{\widetilde{\omega},h_E\otimes h_{\varepsilon_0}}$ such that $$\Delta''u=0\,\, (\,\,\text{or}\,\, \overline{\partial}u=\overline{\partial}^{*}u=0),$$
then by the positivity of the curvature operator
$[\sqrt{-1}\Theta_{h_E\otimes h_{\varepsilon_0}}(E\otimes F),\Lambda_{\widetilde{\omega}}]$
and the inequality (\ref{4242}) we conclude that $u=0$. It follows that
$$\mathbb{H}^{n,q}_{(2)}(Y,E\otimes F)_{\widetilde{\omega},h_E\otimes h_{\varepsilon_0}}=0$$
for any $q\geq1$ with $m\geq\min\{n-q+1,r\}$. This completes the proof of Theorem \ref{ND.}.
\end{proof}

\begin{cor}[=Corollary \ref{.Griffiths}] \label{.Griffiths.}
Let $E$ be a holomorphic vector bundle 
over an $n$-dimensional compact \ka manifold $X$ and $F$ a pseudo-effective line bundle on $X$ equipped with a singular hermitian metric $h$ with semi-positive curvature current.
If $E$ is Griffiths positive then 
\begin{equation}
H^n(X,K_X\otimes E\otimes F\otimes \mathcal{I}(h))=0.\nonumber
\end{equation}
\end{cor}
\begin{proof}
Note that: for any holomorphic vector bundle $E$ the Griffiths positivity is equivalent to the Demailly 1-positivity. Thus Corollary \ref{.Griffiths.} follows from Theorem \ref{ND.}
\end{proof}

\begin{cor}[=Corollary \ref{.Nakano}]  \label{.Nakano.}
Let $E$ be a holomorphic vector bundle over a compact \ka manifold $X$ and $F$ a pseudo-effective line bundle on $X$ equipped with a singular hermitian metric $h$ with semi-positive curvature current.
If $E$ is Nakano positive then for any $q\geq 1$ we have
\begin{equation}
H^q(X,K_X\otimes E\otimes F\otimes \mathcal{I}(h))=0.\nonumber
\end{equation}
\end{cor}
\begin{proof} 
Let $\text{dim}X=n$.
Note that: for any holomorphic vector bundle $E$ the Nakano positivity is equivalent to the Demailly $m$-positivity for any integer $m$ with
$m\geq \min\{n,r\}$. But $$\min\{n,r\}\geq\min\{n-q+1,r\}$$ for any $q\geq1$. Thus Corollary \ref{.Nakano.} follows from Theorem \ref{ND.}.
\end{proof}

\begin{thm}[=Theorem \ref{hehe}]\label{hehe.}
Let $E$ be a holomorphic vector bundle of rank $r$ over an $n$-dimensional compact \ka manifold $X$, $m$ a positive integer.
Assume that $D=\sum^t_{i=1}a_iD_i$ is an effective $\mathbb{Q}$-divisor in X with normal crossings and denote $D'=\sum^t_{i=1}(a_i-[a_i])D_i$.
If $E$ is Demailly $m$-positive then for any $q\geq 1$ with $m\geq\min\{n-q+1,r\}$ we have
\begin{equation}
H^q(X,K_X\otimes E \otimes \mathcal{O}(D'))=0.\nonumber
\end{equation}
\end{thm}
\begin{proof}
For the effective $\mathbb{Q}$-divisor $D=\sum^t_{i=1}a_iD_i$ in X with normal crossings we have $$D\otimes \mathcal{I}(D)=\mathcal{O}(D')$$ where $D'=\sum^t_{i=1}(a_i-[a_i])D_i$. Thus Theorem \ref{hehe.} also follows from Theorem \ref{ND.}
\end{proof}

\begin{cor} [=Corollary \ref{hdhd}]\label{hdhd.}
Let $E$ be a holomorphic vector bundle of rank $r$ over an $n$-dimensional compact \ka manifold $X$ and $D=\sum^t_{i=1}a_iD_i$ an effective normal crossing $\mathbb{Q}$-divisor $D$ in $X$ with $0\leq a_i<1$. If $E$ is Demailly $m$-positive then for any $q\geq 1$ with $m\geq\min\{n-q+1,r\}$ we have
\begin{equation}
H^q(X,K_X\otimes E\otimes D)=0.\nonumber
\end{equation}
\end{cor}
\begin{proof}
Let $D=\sum^t_{i=1}a_iD_i$ be an effective normal crossing $\mathbb{Q}$-divisor $D$ in $X$ with $0\leq a_i<1$. Then we have the equation 
$$D'=\sum^t_{i=1}(a_i-[a_i])D_i=\sum^t_{i=1}a_iD_i=D.$$ 
Therefore, Corollary \ref{hdhd.} follows from Theorem \ref{hehe.}.   \end{proof}

\begin{thm} [=Theorem \ref{thm2}]\label{thm2.}
Let $(X,\omega)$ be a compact \ka manifold of dimension $n$ and $L$ be a holomorphic line bundle on $X$ with  a singular hermitian metric $h$ such that $\sqrt{-1}\Theta_{h}(L)\geq\delta\omega$
for some constant $\delta>0$. If $(E,h_E)$ is an hermitian holomorphic vector bundle on $X$ of rank $r$ such that
\begin{equation}\nonumber
\sqrt{-1}\Theta_{h_E}(E)+\tau\omega\otimes Id_E\geq_{m}0
\end{equation}
for some constant $\tau<\delta$, then for any $q\geq 1$ with $m\geq\min\{n-q+1,r\}$ we have
\begin{equation}
  H^{q}(X,K_X \otimes E\otimes L\otimes\mathcal{I}(h))=0. \nonumber
\end{equation}
\end{thm}
\begin{proof} 
Let $h$ be a singular hermitian metric on $L$ such that 
$i\Theta_{h}(L)\geq\delta\omega$ for some constant $\delta>0$.
For the given singular hermitian metric $h$ on $L$,
by the Demailly-Peternell-Schneider equisingular approximation theorem (Lemma \ref{equi.}), there exist singular hermitian metrics $\{h_{\e} \}_{1\gg \e>0}$ on $F$ with the following properties\,$:$
\begin{itemize}
\item[(a)] $h_{\e}$ is smooth on $Y_\varepsilon:=X \setminus Z_{\e}$, where $Z_{\e}$ is a subvariety on $X$.
\item[(b)] $h_{\varepsilon_2}\leq h_{\varepsilon_1} \leq h_F$ for any $0<\varepsilon_1<\varepsilon_2$.
\item[(c)]$\I{h}= \I{h_{\e}}$  (equisingularity).
\item[(d)]$\sqrt{-1} \Theta_{h_{\e}}(F) \geq (\delta -\e )\omega$.
\end{itemize}
Therefore we can choose a singular hermitian metric $h_{\varepsilon_0}\in \{h_{\e} \}_{1\gg \e>0}$ on $F$ with $$0<\e_0<0.5(\delta-\tau)$$ such that $h_{\e_0}$ is smooth on $Y_{\varepsilon_0}:=X \setminus Z_{\e_0}$ and $$\sqrt{-1} \Theta_{h_{\e_0}}(F) \geq (\delta -\e_0 )\omega.$$
It follows that we have the following estimates 
on the open manifold $Y_{\varepsilon_0}$
\begin{equation} 
\begin{split}\nonumber
&\sqrt{-1}\Theta_{h_E\otimes h_{\varepsilon_0}}(E\otimes F)\\
&=\sqrt{-1}\Theta_{h_E}(E)+\sqrt{-1}\Theta_{h_{\varepsilon_0}}(F)\otimes Id_E\\
&\geq_{m} \sqrt{-1}\Theta_{h_E}(E)+(\delta -\e_0 )\omega\otimes Id_E\\
&>_m\sqrt{-1}\Theta_{h_E}(E)+\tau \omega\otimes Id_E
\geq_{m}0\\
\end{split}
\end{equation}
which implies that the hermitian holomorphic vector bundle 
$(E\otimes F,h_E\otimes h_{\varepsilon_0})$
is Demailly $m$-positive on $Y_{\varepsilon_0}$. The rest is essentially the same as that in the proof of Theorem \ref{ND.} and we omit it for simplicity.
\end{proof}

\begin{cor}[=Corollary \ref{thm3}] \label{thm3.}
Let $(X,\omega)$ be a compact \ka manifold and $L$ a holomorphic line bundle on $X$ with a singular hermitian metric $h_L$ such that $i\Theta_{h}(L)\geq\delta\omega$
for some constant $\delta>0$. If $(E,h_E)$ is an hermitian holomorphic vector bundle on $X$ of rank $r$ such that
\begin{equation}\nonumber
\sqrt{-1}\Theta_{h_E}(E)+\tau Id_E\otimes\omega\geq_{Nak}0
\end{equation}
for some constant $\tau<\delta$, then for any $q\geq 1$ we have
\begin{equation}
  H^{q}(X,K_X \otimes E\otimes L\otimes\mathcal{I}(h))=0. \nonumber
\end{equation}
\end{cor}
\begin{proof}
If $(E,h_E)$ is an hermitian holomorphic vector bundle on $X$ of rank $r$ such that
\begin{equation}\nonumber
\sqrt{-1}\Theta_{h_E}(E)+\tau Id_E\otimes\omega\geq_{Nak}0
\end{equation}
for some constant $\tau<\delta$, then we have
\begin{equation}\nonumber
\sqrt{-1}\Theta_{h_E}(E)+\tau\omega\otimes Id_E\geq_{m}0
\end{equation}
for some constant $\tau<\delta$ and for any $q\geq 1$ with $m\geq\min\{n-q+1,r\}$. Therefore, Corollary \ref{thm3.} follows from Theorem \ref{thm2.}.
\end{proof}

\begin{cor} [=Corollary \ref{Nadel}]\label{Nadel.}
Let $(X,\omega)$ be a compact K\"{a}hler manifold and $L$ be 
a holomorphic line bundle on $X$ with a singular hermitian
metric $h$ such that $i\Theta_{h}(L)\geq\delta\omega$
for some constant $\delta>0$. Then for any $q\geq1$ we have
\begin{equation}
  H^{q}(X,K_X \otimes L\otimes\mathcal{I}(h))=0.\nonumber
\end{equation}
\end{cor}
\begin{proof}
It is obvious that Corollary \ref{Nadel.} follows from Corollary \ref{thm3.} if we take the holomorphic vector bundle $E$ to be the trivial line bundle on $X$ and take the constant $\tau$ to be zero in Corollary \ref{thm3.}.
\end{proof}


\end{document}